\documentclass[a4paper, 11pt,psamsfonts]{amsart}
\usepackage{amsmath}
\usepackage{amsthm}
\usepackage{amssymb}
\usepackage{amscd}
\usepackage{amsfonts}
\usepackage{amsbsy}
\usepackage{enumerate}
\usepackage{graphicx}
\usepackage{calc}
\usepackage{xcolor}
\usepackage{setspace}
\usepackage{epsfig,afterpage}
\usepackage[dvips]{psfrag}
\usepackage{indentfirst, latexsym, bm,amssymb}
\usepackage{bbding}
\usepackage{color}

\newtheorem {theorem} {Theorem}
\newtheorem {proposition}{Proposition}
\newtheorem {corollary}{Corollary}
\newtheorem {lemma}{Lemma}

\numberwithin{equation}{section}

\title[First integrals of analytic differential systems]
{Regularity and convergence of local first integrals of analytic differential systems}

\author[X. Zhang]
{Xiang Zhang}
\address{School of Mathematical Sciences, and MOE--LSC, Shanghai Jiao Tong University, Shanghai 200240, P. R. China}
\email{xzhang@sjtu.edu.cn}



\begin{document}

\begin{abstract}
Poincar\'e proved nonexistence of formal first integrals near a nonresonant singularity of analytic autonomous differential systems. In the resonant case with one zero eigenvalue and others nonresonant, there remains an open problem on regularity and convergence of local first integrals.

Here we provide an answer to this problem. The system has always a local $C^\infty$ first integral near the singularity when it is nonisolated. In any finite dimensional space formed by analytic differential systems having the same linear part at the singularity, either all the systems have local analytic first integrals or only the systems in a pluripolar subset have local analytic first integrals.
\end{abstract}
\maketitle

\section{Introduction and statement of the main results}%\label{s1}

For an analytic autonomous differential system defined in a connected open region $\Omega$ of $\mathbb R^n$, it is a well known fact that the system always have $n-1$ functionally independent analytic first integrals defined in a neighborhood of any regular point in $\Omega$. So the system is completely analytically integrable near any regular point. But near a singularity the analytic differential system may have no analytic even no smooth first integrals. The situation will be more intricate, see e.g. \cite{BLV20,Bi79,LPW2012,LV20, RS09,Zhb}. The existence and regularity of first integrals of an analytic differential system are related to
dynamics on algebraic, geometry and topology, see e.g. \cite{BT00,Zh11}.

An analytic autonomous differential system near a singularity can be written in
\begin{equation}\label{e1}
 \dot x=Ax+f(x), \qquad x\in\left(\mathbb R^n,\ 0\right),
\end{equation}
where the dot denotes the derivative with respect to the time $t$, and $f(x)=o(|x|)$ is an $n$--dimensional vector valued analytic function defined in a neighborhood of the origin. Let $\lambda=(\lambda_1,\lambda_2,\ldots,\lambda_n)$ be the $n$--tuple of eigenvalues of $A$. Hereafter, $\mathbb R$ is the set of real numbers, and $\mathbb Z_+$ and $\mathbb Q_+$ are respectively the sets of nonnegative integers and nonnegative rational numbers.

The existence and regularity of first integrals of system \eqref{e1} depend on both $f(x)$ and the eigenvalues of $A$. In case the eigenvalues of $A$ are not resonant, Poincar\'e \cite{Po1897} proved the next result.

\noindent{\bf Theorem A.} \textit{If the $n$--tuple of eigenvalues $\lambda$ do not resonant, i.e.
\[
\langle m,\lambda\rangle\ne 0, \quad m=(m_1,\ldots,m_n)\in\mathbb Z_+^n,\ \ |m|\ge 2,
\]
where $|m|=m_1+\ldots+m_n$ and $\langle\cdot,\cdot\rangle$ denotes the inner product of two vectors, then the analytic differential system \eqref{e1} has neither analytic nor formal first integral in a neighborhood of the origin.}

For a proof of Theorem A, see e.g. \cite{Fu96,SL01}, where Furta \cite{Fu96} and Shi and Li \cite{SL01} also extended the Poincar\'e result to quasihomogeneous and semi--quasihomogeneous differential systems on their local nonintegrability in the analytic or formal sense.
Shi \cite{Sh2007} provided a sufficient condition on the nonexistence of meromorphic or formal meromorphic first integrals of the analytic differential system \eqref{e1} near a rational nonresonant singularity.

On existence of local analytic or meromorphic  first integrals near a singularity, there appeared some necessary conditions on the eigenvalues of linearization of system \eqref{e1} at the singularity. Chen et al \cite{CYZ2008} in 2008 obtained the optimal upper bound on the number of analytic or formal first integrals of the analytic system \eqref{e1}. Cong et al \cite{CLZ2011} in 2011 got the optimal upper bound on the number of meromorphic or formal meromorphic first integrals for system \eqref{e1}, for quasihomogeneous and semi--quasihomogeneous analytic differential systems near a singularity, and for analytic autonomous differential system near a periodic orbit.

On determination of analytic integrability of analytic differential systems, the pioneer work belongs to Poincar\'e, who proved that a planar analytic differential system having a singularity with a pair of imaginary eigenvalues is locally analytically integrable if and only if it is analytically orbitally equivalent to its linear part. This result was extended to general higher dimensional analytic integrable differential systems by Zung \cite{Zu2002} in 2002, and to analytic integrable Hamiltonian systems in the Liouvillian sense by Zung \cite{Zu2005} in 2005 via the method of the torus action, where he proved that any local analytically Liouvillian integrable Hamiltonian system near a singularity is analytically equivalent to its Birkhoff normal form. Zhang \cite{Zh2011p} in 2011 proved the existence of analytic normalization for both general analytic integrable differential system near a singularity, and diffeomorphism near a fixed point, and provided the concrete forms of their analytic integrable Poincar\'e--Dulac normal forms. This result was extended to partially analytic integrable differential systems by Du et al \cite{DRZ2016} in 2016 with an additional condition.

We note that all the above results provide either necessary conditions on the existence and number of analytic (meromorphic, formal or formal meromorphic) first integrals, or on equivalent characterization on existence of an analytic normalization for analytic integrable systems. Whereas on the existence of first integrals with suitable regularity,  there are some general results for two dimensional differential systems,  see e.g. the book \cite{RS09} and the references therein. But for higher dimensional analytic differential systems \eqref{e1} with the singularity having resonant eigenvalues, the results are very few, see e.g. \cite{LLZ2003}.

In this paper the blanket assumptions are the following:
\begin{equation}\label{e2}
\lambda_1=0,\quad \langle \widetilde m, \widetilde \lambda\rangle\ne 0, \quad \widetilde m\in\mathbb Z_+^{n-1},\ \ |\widetilde m|\ge 2,
\end{equation}
where $\widetilde m=(m_2,\ldots,m_n)$ and $\widetilde \lambda=(\lambda_2,\ldots,\lambda_n)$. Li {\it et al} \cite{LLZ2003} provided the conditions on the existence of a formal or an analytic first integral of system \eqref{e1} as stated in the following.

\noindent{\bf Theorem B.} \textit{Assume that system \eqref{e1} is analytic and satisfies \eqref{e2}. Then system \eqref{e1} has a formal first integral in $(\mathbb R^n,\ 0)$ if and only if the singularity  $x=0$ is not isolated. Moreover, if $n=2$ the first integral could be analytic.}

Zhang \cite{Zh2017} pursued the study on analyticity of the local first integral, and got the next results.

\noindent{\bf Theorem C.} \textit{For system \eqref{e1} with $n-1$ nonresonant eigenvalues $\widetilde\lambda$, the following statements hold.
\begin{itemize}
\item[$(a)$] If all the real parts of $\widetilde \lambda$ have the same sign, then system \eqref{e1} has an analytic first integral in $(\mathbb R^n,\ 0)$ if and only if the singularity $ x=0$ is not isolated.
\item[$(b)$] There exist analytic differential systems \eqref{e1} with $\widetilde \lambda$ having both positive and negative real parts, which have no analytic first integrals in $(\mathbb R^n,\ 0)$.
\end{itemize}
}

According to the results in statement $(b)$ of Theorem C, it is natural to appear the next open problems.

\noindent{\bf Problem 1:} \textit{Are there $C^\infty$ first integrals in $(\mathbb R^n,\ 0)$ of system \eqref{e1} under the conditions \eqref{e2}?}

 \noindent{\bf Problem 2:} {\it What is the measure of the set of analytic system \eqref{e1} satisfying the conditions \eqref{e2}, which have analytic first integrals in $(\mathbb R^n,\ 0)$?}

We first answer Problem 1 on existence of $C^\infty$ first integrals.

\begin{theorem}\label{t1}
Under the conditions \eqref{e2}, the analytic differential system \eqref{e1} has a $C^\infty$ first integral in $(\mathbb R^n,0)$ if and only if the singularity at the origin is not isolated.
\end{theorem}

The next one is an answer to Problem 2.
Let $\mathfrak K$ be the set of analytic differential systems of form \eqref{e1} with the same linear part, which satisfies the conditions \eqref{e2}.
\begin{theorem}\label{t2}
Assume that the origin is a nonisolated singularity of all systems in $\mathfrak K$. Let $\mathcal K$ be any finite dimensional subspace of $\mathfrak K$.  The following statements hold.
\begin{itemize}
\item[$(a)$] If $\mathcal K$ contains an element, which has only formal but not analytic first integral in a neighborhood of the origin, then all elements in $\mathcal K$ except perhaps a pluripolar subset has also this property.

\item[$(b)$] If $\mathcal K$ has a nonpluripolar subset whose any element has an analytic first integral in a neighborhood of the origin, then all systems in $\mathcal K$ have this property.
\end{itemize}
\end{theorem}

Recall that the condition on the nonisolated singularity at the origin is necessary for the existence of local analytic or formal first integrals.
A pluripolar set is a subset of $\mathbb C^m$ for some $m\in\mathbb N$, and it is of Lebesgue measure zero, see e.g. \cite{Kl} or \cite{Pe}. A detail definition on pluripolar set will be given in Section \ref{st2}, where we prove Theorem \ref{t2}.

Theorem \ref{t2} has the next consequence.

\begin{corollary}\label{c1}
Let $\mathfrak P$ be the set of polynomial differential systems of form \eqref{e1} with the same linear part and a uniformly bounded degree, which satisfy the conditions \eqref{e2} and have the origin as a nonisolated singularity. Then either $\mathfrak P$ has all its elements having an analytic first integral in a neighborhood of the origin, or $\mathfrak P$ has only a pluripolar subset whose elements can have analytic first integrals near the origin.
\end{corollary}

We remark that the set $\mathfrak P$ is a finite dimensional space formed by the coefficients of the nonlinear monomials of the systems in $\mathfrak P$.

Corollary \ref{c1} solves the open problem 2 for polynomial differential systems \eqref{e1}. Since a pluripolar set has Lebesgue measure zero, if $\mathfrak P$ has a positive Lebesgue measure subset, whose each element has an analytic first integral in a neighborhood of the origin, then all systems in $\mathfrak P$ have analytic first integrals near the origin. In other words, if $\mathfrak P$ has an element which has only divergent formal first integral in a neighborhood of the origin, then generic systems in $\mathfrak P$ have only divergent formal first integrals near the origin.

All analytic differential systems of type \eqref{e1} form an infinitely dimensional space in the sense of Theorem \ref{t2} and Corollary \ref{c1}. In this case, it is still an open problem \textit{whether the generic systems have only divergent first integrals at the origin.}

This paper is organized as follows. The proof of Theorems \ref{t1} is given Section \ref{st1}, and the proofs of Theorem \ref{t2} and Corollary \ref{c1} will be presented in Section \ref{st2}.

\section{Proof of Theorem \ref{t1}}\label{st1}

\noindent {\it Necessity}. It follows from \cite{LLZ2003,Zh2017}. Here for completeness we present its proof. By the assumption of the theorem and \cite[Theorem 1 $(b)$]{DRZ2016} (see the next Proposition \ref{p11}), system \eqref{e1} has a local smooth first integral of the form $F(x)=x_1+o(x)$  in $(\mathbb R^n,0)$. Taking the invertible change of coordinates $v=\Phi(x)$ in $(\mathbb R^n,0)$ with
\[
v_1=F(x),\ \ \mathbf v_2=(v_2,\ldots,v_n)^\tau=(x_2,\ \ldots, \ x_n)^\tau,
\]
with $\tau$ representing the transpose of a matrix, system \eqref{e1} in this new coordinate system becomes
\[
\dot v_1=0,\quad \dot{\mathbf v}_2=B\mathbf v_2+\mbox{h.o.t.}
\]
This last system has the line of singularities $\mathbf v_2=0$, and consequently system \eqref{e1} has a smooth curve fulfilling singularities and passing the origin. The necessity follows.

For readers' convenience we recall \cite[Theorem 1 $(b)$]{DRZ2016}, which will also be used in the proof of Theorem \ref{t2}. For system \eqref{e1} set
\[
\mathcal R:=\{m\in\mathbb Z_+^n| \ \langle m,\lambda\rangle=0\}.
\]

\begin{proposition}\label{p11}
Assume that $\mathcal R$ has $d<n$ $\mathbb Q_+$--linearly independent elements. If system \eqref{e1} has $d$ functionally independent analytic or formal first integrals, then it has $d$ functionally independent first integrals of the form
 \[
    H_1(x)=x^{\alpha_1}+h_1(x),\,\ldots,\, H_{d}(x)=x^{\alpha_{d}}+h_{d}(x),
 \]
where $\alpha_1,\ldots,\alpha_{d}$ are $\mathbb Q_+$--linearly independent elements of $\mathcal R$, and each $h_j(x)$, $j=1,\,\ldots,\, d$,  consists of nonresonant monomials in $x$ of degree larger than $|\alpha_j|$.
\end{proposition}

\noindent{\it Sufficiency}.  Without loss of generality we can assume that the linear part $A$ of system \eqref{e1} is of the form
\[
A=\left(\begin{array}{cc}
0 &  \widetilde{\mathbf 0}\\
{\widetilde{\mathbf 0}}^\tau &  B
\end{array}\right),
\]
where $\widetilde{\mathbf 0}$ is the $n-1$ dimensional row vector with its entries all zeros,  and $B$ is a square matrix of order $n-1$ with the $n-1$ tuple of eigenvalues $\widetilde \lambda$ and is in lower triangular normal form. Under this consideration system \eqref{e1} is of the form
\begin{equation}\label{e1-1}
\begin{split}
\dot x_1&=f_1(x),\\
\dot {\mathbf x}_2&=B\mathbf x_2+\mathbf f_2(x).
\end{split}
\end{equation}
Hereafter we denote by $\mathbf a_2$ the $n-1$ dimensional column vector $(a_2,\ldots,a_n)^\tau$ for any letter $a$.

Let $\mathbf x_2=\mathbf \phi(x_1)$ be the unique solution of the equation $B\mathbf x_2+\mathbf f_2(x)=0$ in a neighborhood of the origin, which is analytic and can be obtained by the implicit function theorem. Since the singularity at the origin is not isolated, it forces that $f_1(x_1,\mathbf \phi(x_1))\equiv 0$ in a neighborhood of the origin.  Taking the coordinate change of variables
\begin{equation}\label{e1-2}
y_1=x_1,\quad \mathbf y_2=\mathbf x_2-\mathbf \phi(x_1),
\end{equation}
system \eqref{e1-1} can be written in an analytically equivalent way as
\begin{equation}\label{e3}
\begin{split}
\dot y_1&=F_1(y)\mathbf y_2,\\
\dot {\mathbf y}_2&=B\mathbf y_2+F_2(y)\mathbf y_2,
\end{split}
\end{equation}
where $ F_1$ is an $n-1$ dimensional row vector valued function and $ F_2$ is a square matrix valued function of order  $n-1$. Since the transformation \eqref{e1-2} is invertible and analytic, to prove Theorem \ref{t1} is equivalent to prove that system \eqref{e3} has a $C^\infty$ first integral in $(\mathbb R^n,0)$.

\noindent{\it Step} 1. {\it Normal form}.  By Poincar\'e--Dulac normal form theorem (see e.g. \cite{Zhb}) and the assumption \eqref{e2}, there exists a near identity distinguished formal change of coordinates
\begin{equation}\label{e4}
y=\widehat \Phi(z)=z+\widehat \phi(z),
\end{equation}
under which system \eqref{e3} is transformed to its distinguished normal form
\begin{equation}\label{e5}
\begin{split}
\dot z_1&=0,\\
\dot {\mathbf z}_2&=B\mathbf z_2+\widehat G_2(z)\mathbf z_2,
\end{split}
\end{equation}
where $\widehat G_2$ is a matrix valued formal series in $z$ of order $n-1$. We remark that in general by Poincar\'e--Dulac normal form theorem the first equation of \eqref{e5} should have the form $\dot z_1=\widehat G_1(z_1)$, with  $\widehat G_1(z_1)$ a formal series in $z_1$, then the nonisolate of the singularity at the origin forces $\widehat G_1(z_1)\equiv 0$, otherwise we take a sufficiently high order cut--off of \eqref{e4} which will induce that the transformed system has the origin as a singularity of finite multiplicity, a contradiction with the nonisolate of the singularity. Since the normalization is chosen to be distinguished, it compels that $\widehat \phi$ consists of nonresonant monomials and so $\widehat \phi(z)=\widehat \phi(\mathbf z_2)$.

By the Borel lemma, there exists a $C^\infty$ function $\Phi(\mathbf z_2)$ such that $\mbox{jet}_0^\infty \Phi(\mathbf z_2)=\widehat \phi(\mathbf z_2)$. Hereafter $\mbox{jet}_0^\infty \Phi$ denotes the Taylor series of $\Phi$ at the origin. Taking the $C^\infty$ change of coordinates $y=z+\Phi(\mathbf z_2)$, system \eqref{e3} is $C^\infty$ equivalent, in a neighborhood of the origin, to
\begin{equation}\label{e6}
\begin{split}
\left(\begin{array}{c}
\dot z_1\\
\dot{\mathbf z}_2\end{array}\right)
&=\left(I+\partial_z\Phi(\mathbf z_2)\right)^{-1}\left(\begin{array}{c}
F_1(y)\\
B+F_2(y)\end{array}\right)\left(\mathbf z_2+\mathbf\Phi_2(\mathbf z_2)\right)\\
&=\left(\begin{array}{c}
 W_1(z)\mathbf z_2,\\
 (B+W_2(z))\mathbf z_2
 \end{array}\right)
\end{split}
\end{equation}
where $W_1$ is an $n-1$ dimensional $C^\infty$ row vector valued function and $W_2$ is a $C^\infty$ square matrix valued function of order $n-1$, $\partial_z\Phi(\mathbf z_2)$ is the Jacobian matrix of $\Phi$ in $z$ and $\mathbf\Phi_2$ is the last $n-1$ components of $\Phi$. Note that $W_1(z)$ is infinitely flat at the origin, i.e. $\mbox{jet}_0^\infty W_1(z)=0$. If we can prove that system \eqref{e6} has a $C^\infty$ first integral in $(\mathbb R^n,0)$, then system \eqref{e3} and consequently system \eqref{e1} has a $C^\infty$ first integral in $(\mathbb R^n,0)$.

Again applying the Borel lemma to $\widehat G_2(z)$ in \eqref{e5} provides a $C^\infty$ square matrix valued function $V_2(z)$ of order $n-1$ such that $\mbox{jet}_0^\infty V_2(z)=\widehat G_2(z)$. Consider
the system
\begin{equation}\label{e7}
\begin{split}
\dot z_1&=0,\\
\dot {\mathbf z}_2&=B\mathbf z_2+V_2(z)\mathbf z_2,
\end{split}
\end{equation}
Clearly system \eqref{e7} has the $C^\omega$ first integral $H=z_1$. If we can prove $C^\infty$ equivalence between the two systems \eqref{e6} and \eqref{e7}, then system \eqref{e6} will have a $C^\infty$ first integral in $(\mathbb R^n,0)$.

%and $\widehat g(y)$ a $n-1$ dimensional vector valued formal series satisfying
%\[
%\widehat g(y_1,\widetilde 0)=0.
%\]
%As proved in \cite{LLZ2003}, it follows from the assumption that system \eqref{e1} has its singularity $x=0$ nonisolated that
%\[
%\widehat g_1(y_1)\equiv 0.
%\]
%Otherwise, we can take a sufficient higher cut--off of the formal change \eqref{e3} (which is analytic) under which system \eqref{e1} is transformed to
%\[
%\dot y_1=a_m y_1^m+\mbox{\rm h.o.t.}, \quad \dot{\widetilde y}=B\widetilde y+\mbox{\rm h.o.t.}
%\]
%Obviously, this last system has its singularity $x=0$ isolated and with multiplicity $m$, a contradiction.

\noindent{\it Step }2. {\it $C^\infty$ equivalence between systems \eqref{e6} and \eqref{e7}.}

Denote by $\mathcal W$ the vector field associated to system \eqref{e6} and by $\mathcal V$ the vector field associated to system \eqref{e7}.
Observing that systems \eqref{e6} and \eqref{e7} both have the $z_1$--axis as their center manifold, and that $\mathcal W-\mathcal V=0$ when restricted to the center manifold $\mathbf z_2=0$. Moreover, the last construction shows that $\mathcal W-\mathcal V$ is infinitely flat at the origin.

The condition \eqref{e2} implies that all the eigenvalues of $B$ have nonvanishing real parts. By the stable manifold theorem system \eqref{e6} has an $s$--dimensional $C^\infty$ stable manifold $M_w^s$ and a $u$--dimensional $C^\infty$ unstable manifold $M_w^u$, and
system \eqref{e7} has also an $s$--dimensional $C^\infty$ stable manifold $M_v^s$ and a $u$--dimensional $C^\infty$ unstable manifold $M_v^u$. Taking the stable manifold $M_w^s$ and unstable manifold $M_w^u$ as coordinate spaces, and denoting them by $S$ and $U$ respectively. Denote by $C$ the $1$--dimensional center manifold, i.e. the $z_1$--axis. Then $\mathbb R^n=C\oplus S\oplus U$. Denote by $z_1,\ z_s$ and $z_u$ the coordinates on $C$, $S$ and $U$, and we still use $z$ to represent $(z_1, z_s, \ z_u)$.

Note that all the above manipulations are local. Take $\varphi:\ \mathbb R\rightarrow [0,1]$ be a $C^\infty$ function satisfying $\varphi(s)=1$ for $s\in[0,\ 1/2]$ and $\varphi(s)=0$ for $s\ge 1$. Set
\[
\widetilde{\mathcal W}(z)=\left(\begin{array}{c}
\sigma^{-1}\varphi(\|z\|)W_1(\sigma z)\mathbf z_2\\
\left(B+\sigma^{-1}\varphi(\|z\|)W_2(\sigma z)\right)\mathbf z_2
\end{array}\right),
\]
and
\[
\widetilde{\mathcal V}(z)=\left(\begin{array}{c}
0\\
\left(B+\sigma^{-1}\varphi(\|z\|)V_2(\sigma z)\right)\mathbf z_2
\end{array}\right),
\]
where $\sigma>0$ is a parameter, and it will be taken to be suitably small. We have the following facts:
\begin{itemize}
\item  $\widetilde{\mathcal W}$ and $\widetilde{\mathcal V}$ are two globally defined vector fields in $\mathbb R^n$.
 \item In the region $\|z\|\le 1/2$, after the linear change of coordinates $\sigma z\rightarrow z$  the vector fields $\widetilde W$ and $\widetilde V$ are respectively transformed to $W$ and $V$.
\end{itemize}
 So, to prove the $C^\infty$ equivalence between \eqref{e6} and \eqref{e7} is equivalent to verify the $C^\infty$ equivalence between the vector fields $\widetilde{\mathcal W}$ and $\widetilde{\mathcal V}$.

For doing so, set
\[
R(z)=\widetilde{\mathcal W}(z)-\widetilde{\mathcal V}(z).
\]
{\bf Claim 1:} $R$ has a $C^\infty$ decomposition as
\[
R(z)=R^+(z)+R^-(z),
\]
with $R^+,\ R^-\in C^\infty(\mathbb R^n)$ and satisfying
\[
\begin{split}
\partial^\alpha R^+(z)&=0, \quad z\in C\oplus U,\\
\partial^\alpha R^-(z)&=0, \quad z\in C\oplus S.
\end{split}
\]
for all $\alpha=(\alpha_1,\ldots,\alpha_n)\in \mathbb Z_+^n$, where $\partial^\alpha R= \dfrac{\partial^\alpha R}{\partial z_1^{\alpha_1}\partial z_2^{\alpha_2}\ldots\partial z_n^{\alpha_n}}$. Indeed, taking
\[
r_\alpha(z)=\left\{
\begin{array}{cl}
\partial^\alpha R(z), \quad & z\in C\oplus S,\\
0, \quad & z\in C\oplus U,
\end{array}
\right.
\]
Since $\partial^\alpha R(z)=0$ on $C$ for all $\alpha\in\mathbb Z_+^n$, it follows that there exists a $C^\infty$ function $R^+$ such that $\partial^\alpha R^+(z)=r_\alpha(z)$ for $z\in (C\oplus S)\cup (C\oplus U)$.
Here we have used the next facts (see e.g. \cite[Lemma 1.14]{Li2000}):

{\it For two linear subspaces $V$ and $W$ in $\mathbb R^n$ satisfying $V+W=\mathbb R^n$, and two functions $g,h\in C^\infty(\mathbb R^n,0)$ verifying
\[
\frac{\partial^\alpha g}{\partial z^\alpha}(z)=\frac{\partial^\alpha h}{\partial z^\alpha}(z), \quad \mbox{for all }\ z\in V\cap W, \ \ \  \alpha\in\mathbb Z_+^n,
\]
then there exists a $C^\infty$ function $F$ defined on a neighborhood of the origin such that
\[
\frac{\partial^\alpha F}{\partial z^\alpha}(z)=\left\{
\begin{split}%{ll}
\dfrac{\partial^\alpha g}{\partial z^\alpha}(z), &\qquad z\in V,\\
\dfrac{\partial^\alpha h}{\partial z^\alpha}(z), &\qquad z\in W,
\end{split}
\right. \quad  \mbox{ for all } \ \  \alpha\in\mathbb Z_+^n.
\]}

\noindent Set $R^-=R-R^+$. Then
$\partial^\alpha R^-= 0$ on $C\oplus S$ for all $\alpha\in\mathbb Z_+^n$. This proves the claim.

In order to prove the $C^\infty$ equivalence between $\widetilde{\mathcal W}$ and $\widetilde{\mathcal V}$, we only need to prove the $C^\infty$ equivalence between $\widetilde{\mathcal V}$ and $\widetilde{\mathcal V}+R^-$ and  the $C^\infty$ equivalence between $\widetilde{\mathcal V}+R^-$ and $\widetilde{\mathcal W}\ $$(=\widetilde{\mathcal V}+R^-+R^+)$. Next we prove the former. The latter follows from similar arguments as those in the proof of the former.

For proving the $C^\infty$ equivalence between $\widetilde{\mathcal V}$ and $\widetilde{\mathcal V}+R^-$, we adopt the homological method (see e.g.  \cite[Chapter 1, Section 4]{IY2008} or \cite[Section 2.2]{Li2000}). Set
\[
\mathcal V_s=\widetilde{\mathcal V}+s R^-,\quad s\in[0,1].
\]
The homological method shows that if the homological equation
\begin{equation}\label{e8}
[h,\mathcal V_s]=R^-,
\end{equation}
has a $C^\infty$ vector valued solution $h$, then the vector fields $\mathcal V$ and $\mathcal V+R^-$ is $C^\infty$ equivalent, where $[\cdot,\ \cdot]$ is the Lie bracket of two vector fields. Recall that for any two smooth $n$--dimensional vector fields $\mathcal A$ and $\mathcal B$ in variables $z$, their Lie bracket is
\[
[\mathcal A,\ \mathcal B]:=\partial_z\mathcal B \mathcal A-\partial_z\mathcal A\mathcal B,
\]
which is again a vector field, and
\begin{equation}\label{e9}
[\mathcal A,\ \mathcal B]=-\left.\frac{d}{dt}\mathcal B_t\right|_{t=0},
\end{equation}
with $\mathcal B_t:=\left(\psi_{\mathcal A}^t\right)_*\mathcal B$, where $\psi_{\mathcal A}^t$ is the flow of the vector field $\mathcal A$ and $\left(\psi_{\mathcal A}^t\right)_*$ is the tangent map of the flow.

Let $\psi_t$ be the flow of the vector field $\mathcal V_s$ satisfying $\psi_0(z)=z$, and let $P(z)$ be the time $1$ map of the flow, i.e. $P(z)=\psi_1(z)$. Set $\Phi(t,z)=\partial_z\psi_t(z)$. By the fundamental theory of ordinary differential equations on variational equations, it follows that $\Phi(t,z)$ is the fundamental solution matrix of the matrix valued variational equation
 \[
 \frac{dZ}{dt}=\partial_w\mathcal V_s(w)|_{w= \psi_t(z)} Z, \qquad Z(0)=I,
 \]
 of the differential system $\dot z=\mathcal V_s$ along its solution $z(t)=\psi_t(z)$, where $I$ is the identity matrix. Set
\begin{equation}\label{e99}
h(z)=\int_{0}^{-\infty}\Phi^{-1}(s,z) R^-(\psi_s(z))ds,
\end{equation}
and according to \eqref{e9}, set $h_t(z)=\left(\psi_t\right)_*h(z)$. Direct calculations verify that if the series $h$ defined in \eqref{e99} is convergent, then one has
\begin{equation}\label{e9-1}
\begin{split}
h_t(z)&= \frac{\partial \psi_t}{\partial z} h\circ \psi_{-t}(z)=\Phi(t,\psi_{-t}(z))h(\psi_{-t}(z))\\
&=\Phi(t,\psi_{-t}(z)) \int_{0}^{-\infty}\Phi^{-1}(s,\psi_{-t}(z))R^-(\psi_s(\psi_{-t}(z)))ds \\
&=\Phi^{-1}(-t,z) \int_{0}^{-\infty}\Phi^{-1}(s,\psi_{-t}(z))R^-( \psi_{s-t}(z))ds \\
&= \int_{0}^{-\infty}\Phi^{-1}(s-t,z) R^-( \psi_{s-t}(z))ds \\
&= \int_{-t}^{-\infty}\Phi^{-1}(\tau,z)R^-( \psi_{\tau}(z))d\tau,
\end{split}
\end{equation}
where in the fourth equality we have used the identity
\[
I=\frac{\partial}{\partial z}\left(\psi_t\circ \psi_{-t}(z)\right)=\left.\frac{\partial \psi_t(\nu)}{\partial \nu} \right|_{\nu=\psi_{-t}(z)}\frac{\partial}{\partial z} \psi_{-t}(z) =\Phi(t,\psi_{-t}(z))\Phi(-t,z) ,
\]
and in the fifth equality we have applied the identity
\[
\Phi(s-t,z)=\frac{\partial \psi_{s-t}(z)}{\partial z}=\left.\frac{\partial \psi_s(\nu)}{\partial \nu}\right|_{\nu=\psi_{-t}(z)}
\frac{\partial\psi_{-t}(z)}{\partial z}=\Phi(s,\psi_{-t}(z))\Phi(-t,z).
\]
Then we get from \eqref{e9-1} that
\[
\left.\frac{d}{d t}h_t(z)\right|_{t=0}=\Phi^{-1}(0,z)R^-( \psi_{0}(z))=R^-(z).
\]
This together with \eqref{e9} verifies that $h$ is a solution of the homological equation \eqref{e8}.

The remaining is to prove the convergence of the integral defining $h$ and its $C^\infty$ smoothness. Write the integral in \eqref{e99} as
\[
h(z)=\sum\limits_{j=0}\limits^\infty\int_{-j}^{-j-1}\Phi^{-1}(s,z) R^-(\psi_s(z))ds.
\]
Direct calculations yield
\begin{align*}
\int_{-j}^{-j-1}&\Phi^{-1}(s,z) R^-(\psi_s(z))ds \\
&=\left(\partial_zP^{-j}(z)\right)^{-1}\int_{0}^{-1}\Phi^{-1}(s,P^{-j}(z)) R^-(\psi_s(P^{-j}(z)))ds,
\end{align*}
where we have used
\[
\Phi(s,z)=\frac{\partial \psi_s(z)}{\partial z}=\frac{\partial}{\partial z}\left(\psi_{s+j}\circ \psi_{-j}(z)\right)
=\Phi(s+j,P^{-j}(z))\partial_zP^{-j}(z).
\]
Hence $h(z)$ can be written as
\begin{equation}\label{e9-0}
h(z)%=\sum\limits_{j=0}\limits^\infty(P^{-j})_*G(z)
=\sum\limits_{j=0}\limits^\infty\left(\partial_zP^{-j}(z)\right)^{-1}G\left(P^{-j}(z)\right),
\end{equation}
where
\[
G(w)=\int_0^{-1}\Phi^{-1}(s,w)R^-(\psi_s(w))ds.
\]
Observe that to prove the $C^\infty$ smoothness of $h$ is equivalent to prove that $h$ is $C^k$ for any $k\in\mathbb N$.

For any given $k\in\mathbb N$, by the construction of the vector fields $\mathcal V_s$ it follows that
\begin{equation}\label{e10}
{\mathcal V_s}\overset{\sigma\rightarrow 0}{\longrightarrow} Az\quad \mbox{ and } \quad \mathcal V_s(z)=Az \ \ \ \ \mbox{for}\ \ \ \ \|z\|\ge 1,
\end{equation}
and consequently
\begin{equation}\label{e10-1}
\lim\limits_{\sigma\rightarrow 0}P^{-1}(z)=e^{-A}z\quad \mbox{and} \quad \|P^{-1}\|\le e^{\|A\|}+1.
\end{equation}
Hereafter, for a matrix or a matrix function $W(z)$ we denote by $\|W(z)\|$ the maximum of the absolute values of its elements, and
\[
\|W\|:=\max\limits_{\|z\|\le 1}(\|W(z)\|/\|z\|).
\]
According to \eqref{e10} and \eqref{e10-1} together with the definition of $\widetilde{\mathcal W},\ \widetilde{\mathcal V}$ and $\mathcal V_s$, we can choose $\sigma>0$ suitably small such that
\begin{equation}\label{e11-0}
\max\limits_{0\le |m|\le k}\{\|\partial^m P^{-1}\|,\ m\in\mathbb Z_+^n\}\le N_0= 2e^{\|A\|}.
\end{equation}
For proving the convergence of all the partial derivatives up to $k$th order of the series $h$ in \eqref{e9-0}, we need to estimate $\|\partial^m P^{-j}\|$ for $j\ge 1$, and all $m\in\mathbb Z_+^n$ and $0\le |m|\le k$.

\noindent{\bf Claim 2.} There exists a constant $K_k\ge 1$ depending on $k$ such that
\begin{equation}\label{e11}
\max\limits_{0\le |m|\le k}\{\|\partial^m P^{-j}\|\}\le K_k^j,\quad  j\in\mathbb N.
\end{equation}
Indeed, for $j=1$ the inequality \eqref{e11-0} shows that $K_k=N_0$ satisfies the requirement.
For $j>1$ we adopt induction. Assume that \eqref{e11} holds for all positive integers no more than $j$, we consider $j+1$. For any $m\in\mathbb Z_+^n$ satisfying $0\le |m|\le k$,
\begin{equation}\label{e11-1}
\begin{split}
\partial^mP^{-j-1}(z)=&\partial^m\left(P^{-1}\circ P^{-j}(z)\right)\\
=&\left.\frac{\partial P^{-1}(\mu)}{\partial \mu}\right|_{\mu=P^{-j}(z)}\partial^mP^{-j}(z)
+P_m(z),
\end{split}
\end{equation}
where
\[
P_m(z)=\Lambda_m\left(\left.\frac{\partial^{\ell} P^{-1}(\mu)}{\partial \mu^{\ell}}\right|_{\mu=P^{-j}(z)},\partial^\rho P^{-j}(z) ,
2\le |\ell|\le|m|,\ 1\le |\rho|<|m| \right),
\]
with $\Lambda_m$ an $n$--dimensional vector valued polynomial in its variables with nonnegative coefficients. For the given $m$, let $p(m)$ be the maximum number of monomials in the $n$ components of $\Lambda_m$, and set $p_k=\max\limits_{0\le |m| \le k}p(m)$. Then for $m\in\mathbb Z_+^n$ and $0\le |m|\le k$ one has
\begin{align*}
\|P_m(z)\|&\le p(m)\left(\max\limits_{2\le |\ell|\le|m|} \left\|\left.\frac{\partial^{\ell} P^{-1}(\mu)}{\partial \mu^{\ell}}\right|_{\mu=P^{-j}(z)}\right\|\right)^k\left(\max\limits_{1\le |\rho|< |m|}\|\partial^\rho P^{-j}(z)\|\right)^k\nonumber\\
&\le p_k N_0^k \left(\max\limits_{1\le |\rho|< |m|}\|\partial^\rho P^{-j}(z)\|\right)^{2k}.
\end{align*}
This together with \eqref{e11-1} and the induction assumption gives that
\begin{align*}
\|\partial^mP^{-j-1}(z)\|&\le N_0 \left(\max\limits_{1\le |\rho|\le |m|}\|\partial^\rho P^{-j}(z)\|\right)^{2}+
p_k N_0^k \left(\max\limits_{1\le |\rho|\le |m|}\|\partial^\rho P^{-j}(z)\|\right)^{2k}\nonumber\\
&\le 2 p_k N_0^k \left(\max\limits_{1\le |\rho|\le |m|}\|\partial^\rho P^{-j}(z)\|\right)^{2k}\\
&\le 2p_kN_0^k\left(K_{|m|}^{2k}\right)^j\le K_k^{j+1},\nonumber
\end{align*}
where $K_k=\max\{2p_kN_0^k, K_{|m|}^{2k}\}$, and in the last inequality we have used the induction assumption on $j$. Note that $p_k$ depends only on the dimension of the system and the order $k$ of the derivative, and $N_0$ depends on $A$ and $K_{|m|}$ depends on $k$ and $A$. Consequently $K_k$ depends only on $k$. The claim follows.

Finally we prove that $h(z)$ in \eqref{e9-0} is $C^k$ for any $k\in\mathbb N$.

\noindent{\bf Claim 3.} For any $m\in\mathbb Z_+^n$ satisfying $|m|=k$, the series
\begin{equation}\label{e12}
 \sum\limits_{j=0}\limits^\infty\partial^m\left(\left(\partial_zP^{-j}(z)\right)^{-1}G\left(P^{-j}(z)\right)\right)
\end{equation}
is absolutely convergent. For proving this claim we need the next facts:
\begin{enumerate}[$(f_1)$]
\item  Derivatives of product of two functions
\begin{align*}
\partial^m&\left(\left(\partial_zP^{-j}(z)\right)^{-1}G\left(P^{-j}(z)\right)\right)\\
&=\sum\limits_{\begin{subarray}{c} \ell\in\mathbb Z_+^n\\ \ell \prec m\end{subarray}}c_\ell
\partial^\ell\left(\left(\partial_zP^{-j}(z)\right)^{-1}\right)
\partial^{m-\ell}G\left(P^{-j}(z)\right),
\end{align*}
where $c_\ell$'s are nonnegative integers depending on the expansion, and $\ell\prec m$ represents that any element of $\ell$ is no more than the corresponding one of $m$. Set $c_k=\max\{c_\ell;\ |\ell|\le k,\ \ell\in\mathbb Z_+^n\}$. Since $|m|\le k$, the terms in the summation are finite, denote by $s_k$ the maximum number.
\item Taking $\mu_M$ to be one half of the minimum of the positive real parts of eigenvalues of $B$, and $\mathcal P_U(z)$ the projection of $z\in\mathbb R^n$ to the unstable subspace $U$. Then
    \[
    \|\mathcal P_U\psi_{-t}(z)\|\le \kappa e^{-t\mu_M}\|\mathcal P_Uz\|=\kappa e^{-t\mu_M}\|z_u\|,
    \]
where $\kappa$ is a positive constant.
\item Restricting to the center--stable manifold
\[
\mbox{jet}_z^\infty G(z)=0,\quad \mbox{for arbitrary }\ z\in C\oplus S,
\]
with $G$ given in \eqref{e9-0}, where we have used the property that $R^-$ satisfies, i.e. $\mbox{jet}_z^\infty R^-(z)=0$ for arbitrary $z\in C\oplus S$.
\end{enumerate}

By $(f_3)$ it follows that for any $m\in\mathbb Z_+^n$, $|m|\le k$ and any $N\in\mathbb N$
\begin{equation}\label{eii1}
\|\partial^m G(z)\|\le a_m\|z_u\|^N,
\end{equation}
where $a_m$ is a positive constant depending only on $m$. Hence, by $(f_2)$ one has
\begin{equation}\label{e12-0}
\|\partial^{m-\ell}G\left(P^{-j}(z)\right)\|\le a_{m-\ell}\|\mathcal P_UP^{-j}(z)\|^N\le a_{m-\ell}\kappa e^{-jN\mu_M}\|z_u\|^N.
\end{equation}

According to $(f_1)$, now we need to estimate $\left\|\partial^\ell\left(\left(\partial_zP^{-j}(z)\right)^{-1}\right)\right\|$. In fact, we prove that there exists a positive constant $M_k$ depending only on $k$ such that for any $|\ell|\le k$
\begin{equation}\label{e12-1}
\left\|\partial^\ell\left(\partial_zP^{-j}(z)\right)^{-1}\right\|\le M_k^j.
\end{equation}
Its proof is similar to that of \eqref{e11}. We present a sketch one via induction on $|\ell|\ge 1$ and the identity
\[
 \partial^\ell\left(\left(\partial_zP^{-j}(z)\right)\left(\partial_zP^{-j}(z)\right)^{-1}\right)=\partial^\ell(I)=0.
\]
For $|\ell|=1$ one has
\[
\partial^\ell\left(\left(\partial_zP^{-j}(z)\right)^{-1}\right)=-\left(\partial_zP^{-j}(z)\right)^{-1}
\partial^\ell\left( \partial_zP^{-j}(z) \right)\left(\partial_zP^{-j}(z)\right)^{-1}.
\]
Note from \eqref{e10-1} that
\[
\lim\limits_{\sigma\rightarrow 0}\partial_zP^{-j}(z)=e^{-jA}.
\]
So there exists a $\sigma>0$ small enough such that
\[
\left\|\left(\partial_zP^{-j}(z)\right)^{-1}\right\|=2e^{j\|A\|}.
\]
Consequently, for $|\ell|=1$ one has
\begin{equation}\label{e12}
\left\|\partial^\ell\left(\left(\partial_zP^{-j}(z)\right)^{-1}\right)\right\|\le 4 e^{2j\|A\|}\left\|
\partial^\ell\left( \partial_zP^{-j}(z) \right)\right\| \le 4 e^{2j\|A\|}K_k^j.
\end{equation}
This proves \eqref{e12-1} for $|l|=1$.
For $\ell >1$, by the identity
\[
\partial^\ell\left(\partial_zP^{-j}(z)\right)^{-1}=-\left(\partial_zP^{-j}(z)\right)^{-1}
\sum\limits_{\begin{subarray}{c}\eta\prec \ell\\ \eta\in\mathbb Z_+^n,\eta\ne 0\end{subarray}}
d_\eta \partial^\eta\left( \partial_zP^{-j}(z) \right)\partial^{\ell-\eta}\left(\partial_zP^{-j}(z)\right)^{-1},
\]
with $d_\eta$ nonnegative numbers, one has
\begin{align}\label{e12*}
&\|\partial^\ell\left(\partial_zP^{-j}(z)\right)^{-1}\|\\
&\le  2\mathfrak d_k e^{j\|A\|}
\sum\limits_{\begin{subarray}{c}\eta\prec \ell\\ \eta\in\mathbb Z_+^n,\eta\ne 0\end{subarray}}
\| \partial^\eta\left( \partial_zP^{-j}(z) \right)\| \|\partial^{\ell-\eta}\left(\partial_zP^{-j}(z)\right)^{-1}\|,\nonumber
\end{align}
where $\mathfrak d_k$ is the maximum of all $d_\eta$'s with $\eta\prec \ell$ and $|\ell|\le k$. For given $k$ the number of elements in the summation of \eqref{e12*} is always bounded. Hence applying \eqref{e11} and \eqref{e12-1} inductively to \eqref{e12*} yields that there exists a positive constant $M_k$ such that the estimation \eqref{e12-1} holds.

Combining \eqref{e12-0} and \eqref{e12-1}, and taking $N$ in \eqref{eii1} suitably large, one gets from $(f_1)$ that
\[
\begin{split}
&\left\|\partial^m\left(\left(\partial_zP^{-j}(z)\right)^{-1}G\left(P^{-j}(z)\right)\right)\right\|\\
&\le\sum\limits_{\begin{subarray}{c} \ell\in\mathbb Z_+^n\\ \ell \prec m\end{subarray}}c_\ell
\left\|\partial^\ell \left(\partial_zP^{-j}(z)\right)^{-1} \right\|
\left\|\partial^{m-\ell}G\left(P^{-j}(z)\right)\right\|\\
&\le c_ks_kM_k^j a_{m-\ell}\kappa e^{-jN\mu_M}\|z_u\|^N\le \varrho^j \|z_u\|^N,
\end{split}
\]
where $s_k$ numerates the maximum number of elements in the summation for all $m$ with $|m|\le k$, and in the last inequality we have used the fact that $\lim\limits_{N\rightarrow \infty}e^{-N\mu_M}=0$.
Note that we can choose $N$ large enough such that $\varrho<1$.
This shows that for $m\in\mathbb Z_+^n$ and $|m|=k$, the $m$th order derivative of the series \eqref{e9-0} is absolutely convergent. The claim is proved.

Claim 3 verifies that $h(z)$ is a $C^k$ function for arbitrary $k\in\mathbb N$. Going back to $\mathcal W$ and $\mathcal V$ in the region $\|z\|\le 1/2$ via the linear change $\sigma z\rightarrow z$ as mentioned previously, we get a $C^k$ equivalence between $\mathcal W$ and $\mathcal V$ in a suitable neighborhood of the origin, whose size depends on the small positive  $\sigma$. Of course, with increase of $k\in\mathbb N$ the region of the origin where the vector fields $\mathcal W$ and $\mathcal V$ are $C^k$ equivalent may decrease.
By the arbitrariness of $k\in\mathbb N$ it follows that $h(z)$ is a $C^\infty$ function. Then we have a $C^\infty$ equivalence between $\widetilde {\mathcal W}$ and $\widetilde{\mathcal V}$. And consequently, the vector fields $\mathcal W$ and $\mathcal V$ are $C^\infty$ equivalent near the origin. As a consequence, system \eqref{e1} has a $C^\infty$ first integral in a neighborhood of the origin.

It completes the proof of the theorem.  \qed

\section{Proofs of Theorem \ref{t2} and Corollary \ref{c1}}\label{st2}

According to the proof of the sufficiency of Theorem \ref{t1}, system \eqref{e1} is analytically equivalent to system \eqref{e3}. So we only need to prove the theorem for system \eqref{e3}.
Let $\mathcal X_0$ be the linear vector field defined by the linear part of system \eqref{e3}, and let $\mathcal V$ be the set of all vector fields of form \eqref{e3} with their linear part being the same as $\mathcal X_0$.

\begin{lemma}\label{l11}
Let $\mathcal S$ be an affine finite--dimensional subspace of $\mathcal V$, and set
\[
\mathcal S:=\left\{\mathcal X_0+t_1\mathcal X_1+\ldots+t_m\mathcal X_m|\ t=(t_1,\ldots,t_m)\in\mathbb R^m\right\}.
\]
where $\mathcal X_1,\ \ldots, \ \mathcal X_m$ are linearly independent vector fields of the form
\[
(F_1(y)\mathbf y_2,\ F_2(y)\mathbf y_2)^\tau,
 \]
with $F_1$ and $F_2$ respectively $n-1$--dimensional analytic row vector valued function and analytic matrix valued function of order $n-1$ satisfying $F_1(0)=0$ and $F_2(0)=0$.
Then each vector field in $\mathcal S$ has a formal first integral in $(\mathbb R^n,0)$, whose homogeneous component of degree $k\in\mathbb N$ is a polynomial in $t$ of degree $k-1$ without constant term.
\end{lemma}

In order to prove this lemma we need the next result, see e.g. Bibikov \cite[Lemma 1.1]{Bi79}.
\begin{proposition}\label{p12}
Let $C$ and $D$ be square matrices of order $n$ with respectively the eigenvalues $\alpha=(\alpha_1,\ldots,\alpha_n)$ and $\beta=(\beta_1,\ldots,\beta_n)$, and let $\mathcal H^\ell(x)$ be $n$--dimensional vector valued homogeneous polynomials of degree $\ell$. Then the linear operator
\[
\mathcal L(F)(x):=\frac{\partial F(x)}{\partial x} Cx-DF(x), \qquad F(x)\in \mathcal H^\ell(x),
\]
has the set of eigenvalues
\[
\left\{\langle m,\alpha\rangle-\beta_j|\ m\in\mathbb Z_+^n,\ \ |m|=\ell,\ \ j=1,\ldots,n  \right\}.
\]
\end{proposition}

\noindent{\it Proof of Lemma \ref{l11}}. By the construction of $\mathcal S$ it follows clearly that $\mathcal S$ is isomorphic to $\mathbb R^m$. Denote by $\mathcal X_t$ the vector field in $\mathcal S$ for each $t\in\mathbb R^m$, and write its associated differential system as
\begin{equation}\label{e3*}
\begin{split}
\dot y_1&=F_1(t,y)\mathbf y_2,\\
\dot {\mathbf y}_2&=B\mathbf y_2+F_2(t,y)\mathbf y_2,
\end{split}
\end{equation}
with $F_1$ and $F_2$ depending on $t$ linear, and $F_1$ a row vector valued function and $F_2$ a square matrix valued function of order $n-1$.

Priori let $H(t,y)$ be an analytic or a formal first integral of system \eqref{e3*}. By the equivalent characterization of first integrals, it follows that $H(t,y)$ satisfies the next equality.
\begin{equation}\label{e3*1}
F_1(t,y)\mathbf y_2\frac{\partial H}{\partial y_1} +\left\langle B\mathbf y_2+F_2(t,y)\mathbf y_2,\, \frac{\partial H}{\partial \mathbf y_2}\right\rangle\equiv 0.
\end{equation}
Write $H(t,y)$ in series expansion in $y$
\begin{equation}\label{e3*2}
H(t,y)=\sum\limits_{s=\ell}\limits^\infty H_s(t,y),
\end{equation}
with $\ell\in\mathbb N$ and $H_s(t,y)$ a homogeneous polynomial in $y$ of degree $s$ whose coefficients depending on $t$.
Let the Taylor expansions of $F_1(t,y)$ and $F_2(t,y)$ be
\begin{equation}\label{e3*3}
F_1(t,y)=\sum\limits_{s=1}\limits^\infty F_{1s}(t,y), \qquad
F_2(t,y)=\sum\limits_{s=1}\limits^\infty F_{2s}(t,y),
\end{equation}
with each $F_{1s}$ a vector valued homogeneous polynomial of degree $s$ in $y$ and each $F_{2s}(t,y)$ a square matrix valued homogeneous polynomials in $y$ of degree $s$. Here all $F_{1s}$'s and $F_{2s}$'s depend on $t$ linearly.

Let $\mathcal L$ be the linear operator defined by
\[
\mathcal L=\left\langle B\mathbf y_2,\, \frac{\partial }{\partial \mathbf y_2}\right\rangle.
\]
Substituting \eqref{e3*2} and \eqref{e3*3} into \eqref{e3*1}, and equating the homogeneous terms of degree $\ell+s$ for $s=0,1,\ldots$ yields
\begin{align}\label{e3*4}
\mathcal L (H_\ell)=&0,\\
\mathcal L (H_{\ell+s})=&
-\sum\limits_{j=1}\limits^{s}\left\langle F_{2j}(t,y)\mathbf y_2,\frac{\partial H_{\ell+s-j}}{\partial \mathbf y_2}\right\rangle\nonumber\\
&\qquad\  -\sum\limits_{j=1}\limits^{s}  F_{1j}(t,y)\mathbf y_2 \frac{\partial H_{\ell+s-j}}{\partial y_1},\quad \ell=1,2,\ldots\label{e3*5}
\end{align}

Let $\mathcal H_{s}(y)$ be the linear space formed by homogeneous polynomials in $y$ of degree $s$. Since the $n-1$--tuple of  eigenvalues $\lambda_2,\ldots,\lambda_n$ of $B$ are non--resonant, it follows from Proposition \ref{p12} that equation \eqref{e3*4} has only a solution in $y_1$. According to Proposition \ref{p11} we can take $\ell=1$ and $H_1(y)=y_1$.

For $s=1$, equation \eqref{e3*5} is simply
\[
\mathcal L(H_2)=-F_{11}(t,y)\mathbf y_2.
\]
Since the right hand side of this last equation vanishes identically when restricted to $\mathbf y_2=0$, applying Bibikov \cite[Lemma 1.1]{Bi79} to this last equation, one gets a unique homogeneous solution of degree $2$ modulus $y_1^2$, which is linear in both $\mathbf y_2$ and $t=(t_1,\ldots,t_m)$.

For $s=2$, equation \eqref{e3*5} is reduced to
\begin{equation}\label{e3**1}
\mathcal L(H_3)=-\left\langle F_{21}(t,y)\mathbf y_2,\frac{\partial H_2}{\partial\mathbf y_2}\right\rangle
-F_{11}(t,y)\mathbf y_2\frac{\partial H_2}{\partial y_1}-F_{12}(t,y)\mathbf y_2\frac{\partial H_1}{\partial y_1}.
\end{equation}
Note that the right hand side of this last equation is a homogeneous polynomial in $y$ of degree $3$, and also a polynomial in $t$ of degree $2$ without constant term.
Since the right hand side of \eqref{e3**1} vanishes identically when $\mathbf y_2=0$, applying again Proposition \ref{p12}  to equation \eqref{e3**1} associated with $\mathbf y_2$, one gets a unique homogeneous solution in $y$ of degree $3$ modulus $y_1^3$, which is quadratic in $t$.

Inductively, one gets that equation \eqref{e3*5} for each $s=3,4,\ldots$ has a unique homogeneous solution in $y$ of degree $1+s$ modulus $y_1^{1+s}$, which is of degree $s$ in $t$.

This proves the existence of formal first integrals of each vector field $\mathcal X_t$ in $\mathcal S$, with the properties stated in the lemma. \qed

For proving the generic divergence of the first integrals of system \eqref{e3}, we will use the pluripolar set, see \cite{Kl, Ts} for more detail information.

By definition, a {\it pluripolar set} is a subset of $\mathbb C^m$, saying $\Lambda$, which satisfies that
for each $z\in \Lambda$, there exists a neighborhood $U_z$ of $z$ and a plurisubharmonic function $p$ defined on $U_z$ verifying that ${\Lambda}\cap U_z\subset p^{-1}(-\infty)$.

A {\it plurisubharmonic function} is the one: $p:\Omega\rightarrow [-\infty,\infty)$, with $\Omega\subset\mathbb C^m$ an open subset, which
\begin{itemize}
\item is upper semicontinuous, i.e. $\{z\in\Omega: p(z)<c\}$ open for each $c\in\mathbb R$,
\item and is not identically $-\infty$ on any connected component of $\Omega$,
\item and satisfies that for any $z\in \Omega$
\[
p(z)\leq \frac{1}{2\pi}\int_0^{2\pi}p(z+e^{it}r)dt,
\]
with $r\in\mathbb C^m$ such that $z+\sigma
r\in\Omega$, where $\sigma\in\mathbb C$ satisfies $|\sigma|\leq 1$.
\end{itemize}
Recall that a pluripolar set in $\mathbb
C^m$ has Lebesgue measure $0$, and that the countable
union of pluripolar subsets is also pluripolar.

Denote by $\mathcal P$ the set of plurisubharmonic functions in
$\mathbb C^m$ which have minimal growth, i.e. $u(z)-\log\|z\|$ is bounded above when $\|z\|\rightarrow \infty$.
For any subset ${\Lambda}\subset \mathbb C^m$, define
\[
V_{\Lambda}(z)=\sup\{p(z)| \ \, p\in{\mathcal P},\  p|_{\Lambda}\leq 0 \},\qquad z\in\mathbb C^m.
\]
Having the above preparation one can recall the Bernstein--Walsh Lemma, see e.g. \cite{Pe} or \cite[p.156]{Kl}.

\noindent {\bf Bernstein--Walsh Lemma.} {\it Assume that ${\Lambda}\subset
\mathbb C^m$ is not pluripolar, and that $P(z)$ is a polynomial of
degree $k$. Then
\[
|P(z)|\leq \|P\|_{\Lambda}\exp\left(k V_{\Lambda}(z)\right), \quad \mbox{for }\ \ z\in  \mathbb C^m
\]
where $\|P\|_{\Lambda}$ is the supremum of $P$ on $\Lambda$.
}

We now apply the Bernstein--Walsh lemma to prove the next result, which is the key point in the proof of Theorem \ref{t2}.

\begin{lemma}\label{l12}
For the subset $\mathcal S$ given in Lemma \ref{l11}, the following statements hold.
\begin{itemize}
\item[$(a)$] Either each vector field in $\mathcal S$  has an analytic first integral in $(\mathbb R^n,0)$,
\item[$(b)$] or only the vector fields in an exceptional pluripolar subset of $\mathcal S$ have analytic first integrals in $(\mathbb R^n,0)$.
\end{itemize}
\end{lemma}

\begin{proof} If statement $(b)$ holds, the proof is done. We now suppose that $\mathcal S$ has a subset, saying $\mathcal T$, which is not pluripolar, and whose each element has an analytic first integral.

Since we are in the conditions of Lemma \ref{l11}, we will use the notations and conclusions provided in Lemma \ref{l11} and its proof. Recall from Lemma \ref{l11} that any vector field in $\mathcal S$ has a formal first integral, whose homogeneous part of degree $k$ is a polynomial in $t$ of degree $k-1$.

Let $\mathcal B$ be the subset of $\mathbb C^m$ such that each $t\in\mathcal B$ corresponds to the vector field $\mathcal X_t$ in $\mathcal T$. Set ${\mathcal B}=\cup_{r\geq 1}{\mathcal B}_r$, with ${\mathcal B}_r$ the subset of $\mathcal B$ such that for each $t\in\mathcal B_r$ its associated vector field ${\mathcal X}_t$ has a convergent first integral $H_t$ in the ball $D_r$ of radius $1/r$, which is bounded by $1$ in $D_r$. Since $\mathcal B$ is not pluripolar and a countable union of pluripolar sets is again pluripolar, there exists at least a ${\mathcal B}_{r_0}$ which is non--pluripolar.

By the construction of $\mathcal B_{r_0}$ and the proof of Lemma \ref{l11}, it follows that for each $t\in\mathcal B_{r_0}$ the vector field $\mathcal X_t$, and its associated system \eqref{e3*}, has an analytic first integral $H_t(y)$ in $D_{r_0}$, which is bounded by $1$. Write $H_t(y)$ in series in $y$ with coefficients in $t$ as
\[
H_t(y)=\sum\limits_{j\in{\mathbb Z}_+^{n}} H_j(t)y^j,
\]
where $y^j$ is the multi--index as used previously, and $H_j(t)$'s are polynomials of degree at most $|j|-1$ as obtained in the proof of Lemma \ref{l11}.
For each $t\in\mathcal B_{r_0}$, since $H_t(y)$ is analytic in ${\mathcal D}_{r_0}$ and bounded by $1$, one gets from the Cauchy inequality that there exists a $\rho_0>0$ for
which
\begin{equation}\label{e3*7}
R(t):=\sup_{j\in\mathbb Z_+^n} |H_j(t) | \rho_0^{|j|}<\infty,\qquad
t\in {\mathcal B}_{r_0}.
\end{equation}
In fact, $\rho_0$ can be taken to be $1/(r_0+1)$.

According to the values of $R(t)$ in $\mathcal B_{r_0}$, we further decompose  the non--pluripolar set ${\mathcal B}_{r_0}$ in the union of $\mathcal C_s:= \{t\in{\mathcal B}_{r_0}| \,\,R(t)\leq s\},\ \  s\in \mathbb N$. Then there exists at least one non-pluripolar set among $\mathcal C_s$'s, saying $\mathcal C_{s_0}$.
Taking any fixed non--pluripolar compact subset of $\mathcal C_{s_0}$, denoted by $D_0$. Then $R(t)$ is bounded on $D_0$. Moreover, it follows from \eqref{e3*7} that there exists a $\rho_1>0$ such that
\begin{equation}\label{e3*71}
 |H_j(t) | \leq \rho_1^{|j|},\quad \mbox{ for all } \  t\in D_0,\ \ j\in\mathbb Z_+^n.
\end{equation}
Since $H_j(t)$ is a polynomial of degree $|j|-1$, we get from the Bernstein--Walsh lemma together with \eqref{e3*71} that for $t\in\mathbb C^m$
\begin{equation}\label{e3*72}
|H_j(t)|\le\|H_j\|_{D_0}\exp((|j|-1)V_{D_0}(t))\le \rho_1^{|j|} \exp((|j|-1)V_{D_0}(t)),
\end{equation}
where $\|H_j\|_{D_0}$ represents the maximum value of $|H_j(t)|$ on $D_0$.

Choose any fixed compact subset $\mathcal E\subset{\mathbb C}^m$. Then there
exists a $\rho_2>0$ such that
\begin{equation}\label{e3*73}
\exp((|j|-1)V_{D_0}(t))\le \exp((|j|-1)\|V_{D_0}\|_{\mathcal E})\le \rho_2^{|j|},
\end{equation}
where $\|V_{D_0}\|_{\mathcal E}$ represents the maximum value of $|V_{D_0}(t)|$ on $\mathcal E$. Note that for given $D_0$,  $\rho_2$ depends only on $\mathcal E$.
Combining the inequalities \eqref{e3*72} and \eqref{e3*73} one obtains that
\[
\|H_j\|_{\mathcal E}\le \rho_1^{|j|}\rho_2^{|j|}.
\]
Choose any $\rho\in(0, 1/(\rho_1\rho_2))$. Then for any $t\in\mathcal E$ the formal first integral $H_t(y)$ of the vector field $\mathcal X_t$ is convergent in the polydisc
$\left\{ y=(y_1,\ldots,y_n)|\right.$ $ \left. \,|y_j| \le {\rho},i=1,\ldots,n \right\}$.

By arbitrariness of the choice of the compact subset $\mathcal E$ in $\mathbb C^m$, it follows that for any $t\in\mathbb C^m$ the vector field $\mathcal X_t$ has an analytic first integral in a neighborhood of the origin.

It completes the proof of the lemma. \end{proof}

\noindent{\it Proof of Theorem \ref{t2}}. $(a)$ By contrary we assume that there exists a nonpluripolar subset of $\mathcal K$, whose any element has an analytic first integral in a neighborhood of the origin. Then by Lemma \ref{l12} any system in $\mathcal K$ has an analytic first integral in a neighborhood of the origin, a contradiction with the assumption that $\mathcal K$ contains an element, which has only formal first integrals in a neighborhood of the origin.

\noindent $(b)$ follows from the proof of statement $(a)$. This proves the theorem. \qed

\medskip

\noindent{\it Proof of Corollary \ref{c1}}. As mentioned in the remark after Corollary \ref{c1}, $\mathfrak P$ is a finite dimensional space. Let $m$ be the dimension of this space. By the proofs of Theorem \ref{t1} and Lemma \ref{l11} it follows that any system in $\mathfrak P$ has a formal first integral in a neighborhood of the origin, whose each monomial has its coefficient being a polynomial in $t$. We can write the nonlinear parts of the vector fields associated to the systems in $\mathfrak P$ as $t_1\mathcal N_1+\ldots+t_m\mathcal N_m$, where $t=(t_1,\ldots,t_m)\in\mathbb R^m$, and $\mathcal N_j$ for $j\in\{1,\ldots,m\}$ is an $n$--dimensional vector field with one component a nonlinear monomial and others all vanishing. Then the next proof follows from the same arguments as those given in the proof of Lemma \ref{l11}. The details are omitted. \qed

\section*{Acknowledgements}

%The authors appreciate the anonymous  referee for his/her nice comments and suggestions.

The author is partially supported by NNSF of China grant numbers 11671254, 11871334 and 12071284.

\end{document}